\documentclass[final]{article}
\usepackage{geometry}  
\usepackage{amsmath,amssymb,amsfonts,graphicx,amsthm}
\usepackage{epsfig}
\usepackage{epstopdf}
\usepackage{subfigure, float}
\usepackage{algorithmic}
\usepackage{algorithm}
\usepackage[parfill]{parskip}
\usepackage{comment}
\usepackage{cite}
\usepackage{subfloat}

\newtheorem{thm}{Theorem}[section]

\title{Unstructured Geometric Multigrid in Two and Three Dimensions on Complex and Graded Meshes}
\author{Peter R. Brune \footnotemark[2]
  \and Matthew G. Knepley \footnotemark[3]
  \and L. Ridgway Scott \footnotemark[2]}

\begin{document}
\maketitle

\pagestyle{myheadings}
\thispagestyle{plain}
\markboth{BRUNE ET AL.}{GRADED UNSTRUCTURED MG} 

\renewcommand{\thefootnote}{\fnsymbol{footnote}}

\footnotetext[2]{Department of Computer Science, 1100 S. 58th St., University of Chicago, Chicago, IL 60637
  (brune@uchicago.edu, ridg@uchicago.edu).  Partial support from NSF grant DMS-0920960.}

\footnotetext[3]{Computation Institute, 5735 S.  Ellis Ave., University of Chicago, Chicago, IL 60637,
  (knepley@ci.uchicago.edu), Supported by the Office of Advanced Scientific Computing Research, Office of Science, U.S.
  Department of Energy, under Contract DE-AC02-06CH11357.  }

\renewcommand{\thefootnote}{\arabic{footnote}}

\begin{abstract}
  The use of multigrid and related preconditioners with the finite element method is often limited by the difficulty of
  applying the algorithm effectively to a problem, especially when the domain has a complex shape or the mesh has adaptive
  refinement.  We introduce a simplification of a general topologically-motivated mesh coarsening algorithm for use in
  creating hierarchies of meshes for geometric unstructured multigrid methods. The connections between the guarantees of
  this technique and the quality criteria necessary for multigrid methods for non-quasi-uniform problems are noted.  The
  implementation details, in particular those related to coarsening, remeshing, and interpolation, are discussed.
  Computational tests on pathological test cases from adaptive finite element methods show the performance of the
  technique.
\end{abstract}

\section{Introduction}
\label{sec:intro}





Multigrid methods enable the asymptotically optimal approximate numerical solution to a wide array of partial
differential equations.  Their major advantage is that they have, for a wide class of problems, $O(n)$ complexity for
solving a linear system of $n$ degrees of freedom (DoFs) to a specified accuracy.  However, there are many stumbling
blocks to their full implementation in conjunction with adaptive unstructured finite element methods.


Algebraic multigrid (AMG) methods and widely-available libraries implementing them have made great strides in allowing
for the widespread use of multilevel methods in computational science \cite{StubenAMG}.  This is due to their relatively
black-box nature and ability to be used for a wide variety of common problems without much expert tuning.  Interesting
problems often must use meshes that have been adaptively refined in order to resolve numerical error or scale of
interest without regard to multilevel discretization.  In order to develop a fast multilevel method for such a problem,
either an algebraic or geometric coarsening procedure must be employed.  Practitioners faced with such problems are
often immediately drawn to AMG because of the difficulties of geometric coarsening on nontrivial domains.

However, there are definite advantages to geometric multigrid.  In fact, there have been several attempts to integrate
some geometric intuition into algebraic methods \cite{FeuchterMG,ChowMGSmoothness,JonesAMGe}.  For one, efficiency
guarantees for AMG are based upon heuristic, computational observation, or theoretical bounds
that are hard to satisfy computationally.  Geometric multigrid methods, on the other hand, can be proven to have optimal
convergence properties assuming certain conditions on the meshes, which will be elaborated upon in Section~\ref{sec:umg}.



An observation made by Adams \cite{adamsNNMGevaluation} on the state of unstructured geometric multigrid methods was
that while Zhang \cite{ZhangOOMultigridI, ZhangNonNestedII} provided guarantees of the optimality of geometric multigrid
methods, they are not directly applicable to the methods of mesh coarsening and hierarchy construction as they are
typically implemented.  Our goal here is to use computational geometry methods to meet these quality criteria, to show a
way of constructing a set of meshes in an optimal fashion that would satisfy these criteria, and to show that we have
constructed a successful optimal-order multigrid method.


\subsection{Adaptive Finite Elements and Multigrid}
\label{sec:afem}


A fairly typical scheme for constructing a series of meshes for geometric multigrid methods involves uniform refinement
from some coarsest starting mesh.  However, the goals of refinement for the purpose of resolving discretization or
modelling error and refinement for the purposes of creating a multigrid hierarchy can easily be at odds with each other
if care is not taken.  If a mesh has been refined to satisfy some analytical or adaptively determined error control,
then it may have vastly different length scales existing in different parts of the mesh.  It is because of this that the
combination of refinement in order to resolve physics and refinement for the purpose of multigrid is fraught with peril.


If one refines for the purpose of creating a multigrid hierarchy from an initial coarse mesh with some grading to handle
the numerical properties of the problem, the refinement of the finest mesh will not reflect the error properly.
Conversely if some stages of the refinement for the physics are used for the multigrid hierarchy, one may be unable to
guarantee that the meshes would satisfy the quality and size constraints required by multigrid.  Because of this it is
very difficult to make the two concepts, adaptive refinement and geometric multigrid, go hand in hand as a single
adaptive meshing strategy that satisfies the needs of both methods in the general case.


In typical applications in computational science, one is often given a mesh that has already been adapted to the physics
of some problem and asked to solve some some series of linear equations repeatedly on this static mesh.  Two examples
where this happens quite frequently are in optimization problems, and in the solution of nonlinear equations by methods
requiring the equations to be linearized and solved repeatedly.  In this case, the only available geometric multigrid
methods are based upon coarsening, and huge advantages may be reaped from precomputation of a series of coarser spaces
that effectively precondition the problem.

The need for a mesh refined to resolve error can be demonstrated in some fairly straightforward cases.  The need for
\textit{a priori} grading around a reentrant corner to resolve pollution effects \cite{BabuskaDirectInverse} is a
well-studied classical phenomenon \cite{ApelComparison, LiaoPollutionEffects, BlumExtrapolation} in adaptive finite
element computation.  Multigrid computations on reentrant corner problems have been analyzed on simple meshes in the two
dimensional case with shape-regular refinement \cite{chen2010coarsening} or structured grading
\cite{JungMultilevelGraded}.  A mesh arising from these requirements has disproportionately many of its vertices
concentrated around the reentrant corner, such as the refined mesh shown in Figure~\ref{fig:reentrantcorner}. Around a
reentrant corner of mesh interior angle $\theta_r \geq \pi$ we can define a factor
\begin{equation*}
\label{eq:mudefinition}
\mu \geq \frac{\pi}{\theta_r}
\end{equation*}
Then, given some constants $C_a > 0$ and $C_b < \infty$ related to the maximum length scale in the mesh, a mesh that
will optimally resolve the error induced by the reentrant corner, for $h$ as the length scale of the cell
containing any point a distance $r$ away from the corner, will satisfy: 
\begin{equation*}
  \label{eq:cornergrading}
  C_ar^{1 - \mu} \leq h \leq C_br^{1 - \mu}
\end{equation*}
This problem will be considered further in Section~\ref{sec:problems}.
\begin{figure}[H]
  \centering
  \includegraphics[width=140mm]{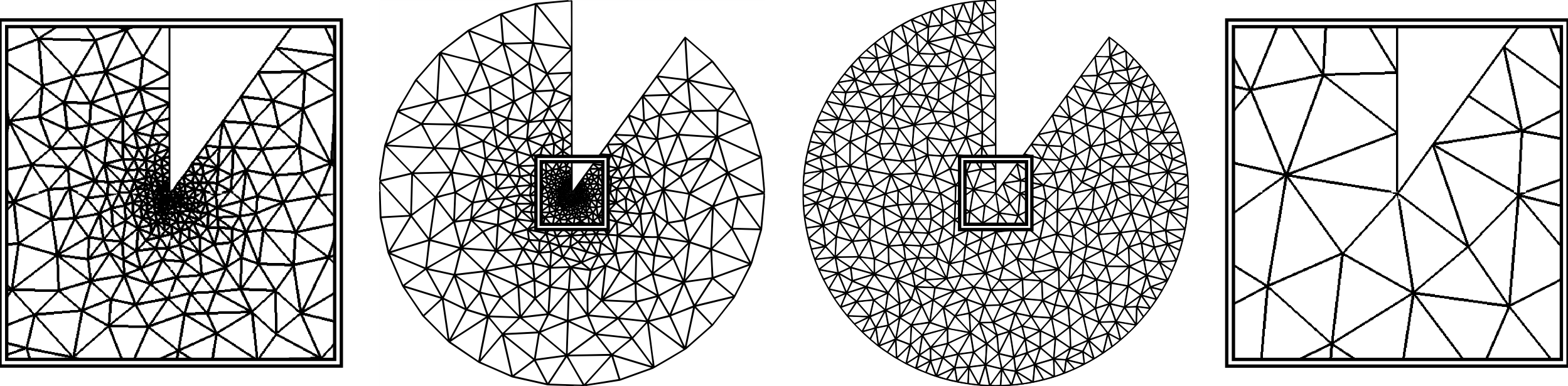}
  \caption{Quasi-uniform (right) and \textit{a priori} (left) refined meshes: 500 vertices}
  \label{fig:reentrantcorner}
\end{figure}
\subsection{Unstructured Geometric Multigrid}
\label{sec:umg}


When the creation of a coarsened hierarchy is considered in the geometric sense, we mean that the problem is reduced to
finding some series of meshes $\{M^0 ... M^n\}$ where $M^0$ is the initial mesh and and $M^n$ is the coarsest mesh.
This hierarchy is constructed starting from $M^0$.  The differential operator is either rediscretized on each mesh in
order to create the series of coarse problems, or interpolators between function spaces on the meshes are used to create
approximate coarse versions of the original fine linear system, a process known as Galerkin multigrid \cite{SmithDD}.
The experiments in this paper are all done by rediscretization.

One restriction that we will maintain throughout this paper is that the meshes are simplicial and node-nested.  This
means that the vertices of mesh $M^{i+1}$ are some subset of the vertices of mesh $M^{i}$.  This is an assumption used
by many mesh refinement and coarsening methods, as it is a reasonable and common assumption that the vertices of the
initial fine mesh adequately capture the domain's geometry.  This is in contrast to fully nested multigrid hierarchies,
where each coarse cell is the union of several fine cells.  A series of non-nested unstructured coarse meshes do not
necessarily need to be node-nested either.  For instance, the mesh hierarchy may be created by repeatedly remeshing some
CAD or exact description of the domain at different resolutions.


There is some justification for extending the classical multigrid convergence results to the case of non-nested meshes
\cite{lonermorganunstructMG, bittencourt1998non, adaptiveNNMGbitten}.  Results of optimal multigrid convergence have
been extended to non-quasi-uniform \cite{ZhangNonNestedII} and degenerate meshes \cite{zhang1995optimal} in two
dimensions.  Three dimensional non-nested multigrid has been proven to work in the quasi-uniform case \cite{RidgProof}.

We will use the mesh quality conditions required for the non-quasi-uniform case \cite{ZhangNonNestedII}.  While the
example mesh hierarchies used in this paper might be quasi-uniform when considered alone, we may not assume
quasi-uniformity independent of the size of the finest mesh in the asymptotic limit of refinement.  It should be noted
that the non-quasi-uniform theory does not exist for three dimensions, but the quasi-uniform theory \cite{RidgProof}
serves as a guide to how the method should perform.

\subsection{Mesh Constraints}

For any mesh $M$ and any cell $\tau$ in $M$ let $\rho_{\tau}$ be $\tau$'s incircle diameter and $h_\tau$ its longest
edge, and define the aspect ratio as
\begin{equation}
  \label{eq:aspectratio}
  AR(\tau) = \frac{h_\tau}{\rho_{\tau}}.
\end{equation}
The quality of approximation \cite{ShewchukGood} and matrix condition number \cite{babuskaangle} in FEM calculations
depend strongly on $AR(\tau)$.  For some constant $C_{AR}$, we must require that the aspect ratio of any given cell in
$\{M_0 ... M_n\}$ satisfy
\begin{equation}
  \label{eq:aspectratiobound}
  AR(\tau) \leq C_{AR}.
\end{equation}
The other half of the criteria are the local comparability conditions.  Assume that if we have two meshes, $M^k$ and
$M^{k-1}$, we define, for some cell $\tau$ in $M^k$
\begin{equation}
  \label{eq:overlapset}
  S_{\tau}^k = \{T : T \in M^{k - 1}, T \cap \tau \neq \emptyset\}
\end{equation}
which defines the set of cells in a mesh $M^{k-1}$ overlapping a single cell $\tau$ in the next coarser mesh $M^k$. We
can state the local comparability conditions as
\begin{subequations}
  \begin{equation}
  \label{eq:loccomp1}
  \sup_{\tau \in M^k}\{|S^k_{\tau}|\} \leq C_o \mbox{\text{ for some }} C_o \geq 0
  \end{equation}
  \begin{equation}
  \label{eq:loccomp2}
  \sup_{\tau \in M^k}\{\sup_{T \in S^k_{\tau}}\{\frac{h_\tau}{h_T}\}\} \leq C_l \mbox{\text{ for some }} C_l \geq 1
  \end{equation}
\end{subequations}
Eq.~\ref{eq:loccomp1} implies that each cell intersects a bounded number of cells in the next finer mesh, and
Eq.~\ref{eq:loccomp2} that the length scale differences of overlapping cells are bounded.  We will also state here, for
the sake of completeness, the assumption from the standard proofs of multigrid convergence necessary for the algorithmic
efficiency of the method.  Define
\begin{equation*}
  \label{eq:cardinalityM}
  |M| = \text{\# of cells in mesh $M$}
\end{equation*}
Then, for some $C_m \geq 2$
\begin{equation}
  \label{eq:sufficientdecrease}
  |M^k| > C_{m}|M^{k+1}|.
\end{equation}
The implication is that at each coarser mesh there must be a sufficient decrease in the number of cells.  A geometric
decrease in the number of cells over the course of coarsening is necessary in order to have an $O(n)$ method.



\section{Mesh Coarsening}
\label{sec:coarsening}


For node-nested coarsening of an unstructured mesh, we separate the process of mesh coarsening into two stages: coarse
vertex selection and remeshing.  The problem of coarse vertex selection on a mesh $M$ involves creating some subset
$V^c$ of the vertices $V$ of $M$ that is, by some measure, coarse.  Several techniques have been described for coarse
vertex selection on unstructured simplicial meshes.  The most widely used of these methods are Maximum Independent Set
(MIS) algorithms \cite{GuillardMultigrid, ChanMultigrid, AdamsUMG}.  These choose a coarse subset of the vertices of the
mesh such that no two vertices in the subset are connected by a mesh edge.  The resulting set of vertices is then
remeshed.  

Given some mesh $M$, we define the graph $G_M = (V, E)$, which consists of the edges $E$ and vertices $V$ of $M$.  The
MIS algorithms may then be described as finding some set of coarse vertices $V^{MIS}_c$ such that
\begin{equation*}
  (v_a, v_b) \not\in E \mbox{ for all pairs of vertices $(v_a, v_b) \in V^{MIS}_c$}.
\end{equation*}
This may be implemented by a simple greedy algorithm where at each stage a vertex is selected for inclusion in
$V^{MIS}_c$ and its neighbors are removed from potential inclusion.  There are a couple of issues with this method.  The
first being that there is no way to determine the size of the mesh resulting from coarsening.  The other is that there
are no real guarantees on the spatial distribution of vertices in the resulting point set.  It has been shown that MIS
methods may, for very simple example meshes, degrade the mesh quality quite rapidly \cite{MillerCoarsen}.  Other methods
for choosing the coarse vertex selection have been proposed to mitigate this shortcoming, often based upon some notion
of local increase in length scale \cite{ReyNodeNested}.


%

\subsection{Function-Based Coarsening}

The coarse vertex selection method we will focus on for this was developed by Miller et al. \cite{MillerCoarsen} and is
referred to as function-based coarsening.  The method begins by defining some spacing function $Sp(v)$ over the vertices
$v \in V$ of $M$.  $Sp(v)$ is some approximation of the local feature size.  In practice, the approximation of the local
feature size based upon the nearest-mesh-neighbor distance at each vertex is a reasonable and efficiently computable spacing
function.

Define $\beta$ as the multiplicative coarsening parameter.  $\beta$ can be considered the the minimum factor by which
$Sp(v)$ is increased at $v$ by the coarsening process.  Here we choose $\beta > \beta_0$ where $\beta_0 = \sqrt{2}$ in
2D and $\beta_0 = \sqrt{3}$ in 3D by the simple fact that $\beta = \beta_0 + \epsilon$ for small $\epsilon$ would
reproduce the repeated structured coarsening of an isotropic, structured, shape-regular mesh where the length scale is
increased by two in every direction.  $\beta$ may also be tuned, changing the size of the set of resulting coarse
vertices in a problem-specific fashion to account for mesh and function-space properties, such as polynomial order.

We say that a coarse subset of the vertices of $M$, $V^{FBC}_c$ satisfies the spacing condition if
\begin{equation}
\label{eq:spacingfunctioncondition}
\beta (Sp(v_i) + Sp(v_j)) < dist(v_i, v_j) \mbox{ for all pairs $(v_i, v_j) \in V^{FBC}_c$}.
\end{equation}
After determining $V^{FBC}_c$, $M^c$ may be created by some remeshing process.  A hierarchy of meshes may then be
created by reapplying the coarsening procedure to each newly-minted coarse mesh with constant $\beta$ and some
recalculated $Sp$ in turn to create a yet coarser mesh.  This may be done until the mesh fits some minimum size
requirement, or until the creation of a desired number of meshes for the hierarchy.

The original authors \cite{MillerCoarsen} had fast numerical methods in mind when proving quality bounds for
function-based coarsening, and a number of the properties required of the meshes are spelled out in the original work.
For one, the maximum aspect ratio of the resulting mesh hierarchy may be bounded by some constant, satisfying
Eq.~\ref{eq:aspectratiobound}.  We also have that, with $h_{i}(x)$ being the length-scale of the cell overlapping the point
$x$ for all $x \in \Omega$:
\begin{equation*}
\frac{1}{\mathcal{I}}h_{i+1}(x) \leq h_{i}(x) \leq \mathcal{I}h_{i+1}(x) \mbox{ for some $\mathcal{I} > 0$.}
\end{equation*}
Note that this condition implies the first of the local comparability conditions (Eq.~\ref{eq:loccomp1}).
By combining the length scale bound with the aspect ratio bound, one may infer that only a certain
number of fine cells may be in the neighborhood of a coarse cell, bounding the number of overlapping cells
(Eq.~\ref{eq:loccomp2}).  Finally, for the coarsest mesh $M_n$ and some small constant $b$,
\begin{equation*}
  |M_n| \leq b.
\end{equation*}
This implies that the coarsening procedure will be able to create a constant-sized coarse mesh.  Because of these
conditions and their parallels with the conditions on the multigrid hierarchy, this method for coarsening is
particularly appealing.

Parts of the function-based coarsening idea have been incorporated into other methods for mesh coarsening by
Ollivier-Gooch \cite{OllivierGoochContraction}.  Some similarities between the method we propose here and this previous
work are that both use traversal of the mesh and local remeshing in order to coarsen meshes that exhibit complex
features.  This method constructs the conflict graph
\begin{equation*}
  \label{eq:conflictgraph}
  G_{C} = (V, E_C)
\end{equation*}
where
\begin{equation*}
  \label{eq:conflictedges}
  E_C = \{(v_i, v_j) : \beta(Sp(v_i) + Sp(v_j)) > dist(v_i, v_j)\}.
\end{equation*}
This graph is then coarsened with an MIS approach as shown above.  Note that in the limit of extreme grading or large
$\beta$ the $G_C$ can grow to be of size $O(|V|^2)$.  Our method avoids this by localizing and simplifying the
notion of the spacing function and choice of vertex comparisons.  We modify function based coarsening in a way that
reliably guarantees optimal complexity irregardless of mesh non-quasi-uniformity as discussed in Section~\ref{sec:afem}
without dependence on the mesh and the parameter $\beta$.

\subsection{Graph Coarsening Algorithm}
\label{sec:decimation}



We describe here a greedy algorithm for determining an approximation $V^{GCA}_c$ to the set of coarse vertices
$V^{FBC}_c$ satisfying a weakened notion of the spacing condition based upon $G_M$.  Talmor \cite{TalmorThesis} proposed
using shortest-weighted-path distance determined by traversing along mesh edges to accomplish this.  Instead of doing
using the shortest-weighted-path approach, we choose to use the edge connectivity to progressively transform $G_M$ into
some final $G_{\tilde{M}^c}$, which approximates the connectivity of the coarse mesh, $M^c$. It should be noted that,
beyond the initial $G_M$ formed at the start of the algorithm, $G_M$ does not necessarily satisfy the condition that its
edges represent the edges of a valid simplicial mesh.  Remeshing given $M$ and $V^{GCA}$ as it is created in this
section is discussed in Section~\ref{sec:remeshing}.

\begin{figure}
  \centering
    \subfigure[$G_M$ with $Sp(v)$] {
      \label{fig:fbc_a}
      \includegraphics[width=50mm]{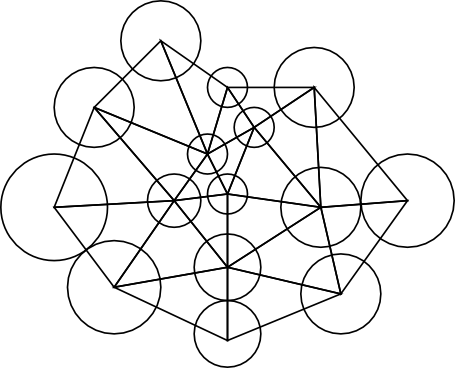}
    }
    \subfigure[$G_M$ with $\beta Sp(v)$] {
      \label{fig:fbc_b}
      \includegraphics[width=50mm]{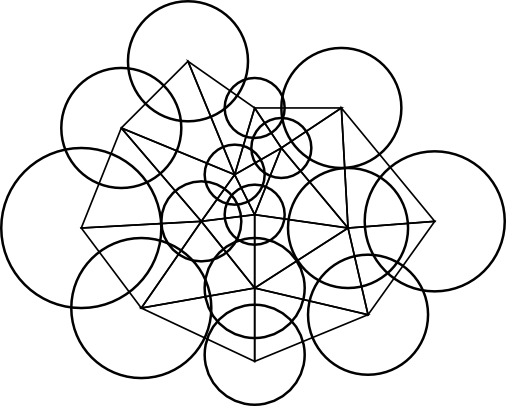}
    }
    \\
    \subfigure[Subalgorithm on $v$] {
      \label{fig:fbc_c}
      \includegraphics[width=50mm]{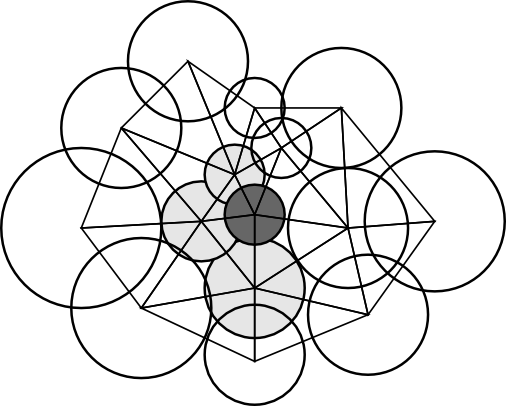}
    }
    \subfigure[Subalgorithm on $v_1$] {
      \label{fig:fbc_d}
      \includegraphics[width=50mm]{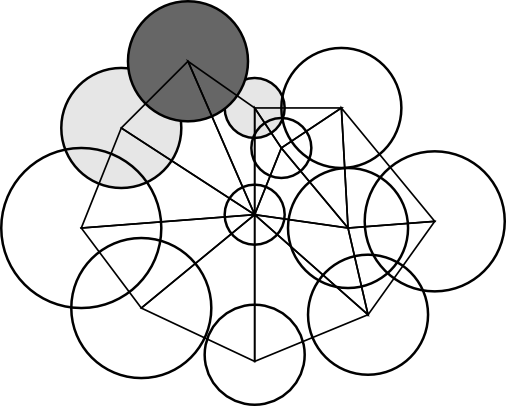}
    }
    \label{fig:fbc}
    \caption{This shows coarsening of a patch of $G_M$.  $Sp$ is defined, multiplied by $\beta$, and the subalgorithm is
      invoked for two vertices, $v$ and $v_1$ shown in Figs. \ref{fig:fbc_c} and \ref{fig:fbc_d} in dark grey with
      vertices violating the graph spacing condition shown in light gray.}

\end{figure}

We begin this process by modifying condition \ref{eq:spacingfunctioncondition} to be
\begin{equation}
  \label{eq:graphspacingcondition}
  \beta (Sp(v_1) + Sp(v_2)) < dist(v_1, v_2) \mbox{ for all $(v_1, v_2) \in E$ of $G_M$.}
\end{equation}
This restricts the spacing condition to take into account only distance between vertices connected by graph edges rather
than all pairs of vertices.  This restriction makes sense considering the calculation of the $Sp(v)$ as nearest-neighbor
distance is based upon adjacency of the mesh.  This is an important simplification for complex domains, where the mesh
may represent complex features of the domain such as holes, surfaces, and interfaces with different local
characteristics on each side that may be hard to encode when the vertices are considered as a point cloud. In our model,
the edges of the mesh are seen as encoding the topological connectivity of the domain.  Spacing only applies to
topologically connected areas of the mesh as encoded in $G_M$.

Define
\begin{equation*}
  F(v) = \mbox{\text{\texttt{unknown}, \texttt{included}, or \texttt{excluded} for all}} v \in V.
\end{equation*}
The algorithm starts with $F(v) = \text{\texttt{unknown}}$ for all $v \in V$.  Vertices that must be included for some
reasons such as the boundary issues described in Section~\ref{sec:boundaries}, are set to be \texttt{included}
automatically and visited first.  The remaining $v \in V$ are then visited in some arbitrary order.  If $F(v) =
\text{\texttt{excluded}}$, the iteration skips the vertex.  If $F(v) = \text{\texttt{unknown}}$, $F(v)$ is changed to
$\text{\texttt{included}}$.

When $F(v)$ is set to \texttt{included}, a subalgorithm is invoked that, given $G_M$ and some particular vertex $v$,
transforms $G_M$ to some intermediate $G_M^v$.  This subalgorithm corresponds to coarsening the area around $v$ until
the spacing function is satisfied for all $N_{G^v_M}(v)$.  Each $w \in N_{G_M}(v)$ are tested to see if the edge $(v,
w)$ violates Eq.~\ref{eq:graphspacingcondition}.  In the case that the condition is violated and $F(w) =
\text{\texttt{unknown}}$, $w$ is removed from $G_M$ and $F(w)$ is changed to $\text{\texttt{excluded}}$.  A removed
$w$'s neighbors in $G_M$ are added to $N_{G_M}(v)$ by adding edges in $G_M$ between $v$ and all $u \in N_{G_M}(w)$ if $u
\not\in N_{G_M}(v)$ already.  This may also be considered as an edge contraction from $w$ to $v$.

The outcome of this subalgorithm, $G^v_M$, has that all $w \in N_{G^v_M}(v)$ such that $F(w) = \text{\texttt{unknown}}$
obey Eq.~\ref{eq:graphspacingcondition}.  There is the possibility that for some $v_1 \in N_{G^v_M}(v)$, $F(v_1) =
\text{\texttt{included}}$ and Eq.~\ref{eq:graphspacingcondition} is not satisfied.  This may arise due to the necessary
inclusion of boundary vertices due to conditions described in Section~\ref{sec:boundaries}.  After the subalgorithm
terminates, one is left with $G^v_M$.  The algorithm then visits some $v_1$ (Figure \ref{fig:fbc_d}), which, if
\texttt{included}, will create some $G^{v_1}_M$ by modification of $G^{v}_M$.  Once every vertex is visited, the whole
algorithm terminates.  At this point we may define
\begin{equation*}
V^{GCA}_c = \{v : v \in V, F(v) = \text{\texttt{included}}\}.
\end{equation*}
Despite the fact that we are considering coarsening and remeshing as a preprocessing step, we still have the goal that
the algorithm should be $O(n)$ with $n$ as the number of DoFs in the discretized system.  For piecewise linear
finite elements, the number of DoFs in the system is exactly the number of vertices, $|V|$, in $M$.  For all reasonably
well-behaved meshes, we can additionally assume that $|E| = O(|V|)$ and $|M| = O(|V|)$.  This
implies that $n = O(|V|)$.  The complexity of a single invocation of the subalgorithm may be $O(|V|
+ |E|)$ if some large fraction of the mesh vertices are contracted to some $v$ in one step.  However, the aggregate
number of operations taken to reach $G_{\tilde{M}^c}$ must be the order of the size of the original graph $G_M$.  This
aggregate complexity bound is independent of the order in which $v$ are visited.  While here we keep the assumed order
in which the vertices are visited arbitrary, we see in Section~\ref{sec:boundaries} that specifing some restrictions on
the order may be used to preserve mesh features.


\begin{thm}
  \label{thm:complexity}
  Given a graph derived from a mesh $G_M$, the graph coarsening algorithm will create $V^{GCA}_c$ in $O(|V| +
  |E|)$ time.
\end{thm}

\begin{proof}
  \label{prf:complexity}
  The fact that the complexity of the algorithm depends on $|V|$ at least linearly is apparent from the fact that every
  vertex is visited once.  In order to show that the entire algorithm is $O(|V| + |E|)$ we must establish that
  each edge $e$ is visited at most a constant number of times independent of the size of the graph.

  As vertex $v$ is visited and $F(v)$ is set to \texttt{included}, the subalgorithm inspects each $e = (v, n)$ for $n
  \in N_{G_M}(v)$ to see if they satisfy the spacing condition with $v$.  $e$ is either deleted from $G_M$ if $n$ and
  $v$ do not satisfy the spacing condition, or left in place if the spacing condition is satisfied or if $F(n) =
  \text{\texttt{included}}$.

  Edges that are deleted are not ever considered again.  Therefore, we must focus on edges that survive comparison.
  Suppose that an edge $e = (v_0, v_1)$ is considered a second time by the subalgorithm at the visit to vertex $v_0$.
  As this is the second visit to consider $e$, $F(v_1)$ must be \texttt{included} as in the first consideration of $e$
  necessarily happened during the visit to $v_1$.  As both endpoints of $e$ have now been visited, there is no way $e$
  may be considered a third time.  As each vertex is visited once and the distance across each edge in $G_M$ is
  considered no more than twice, the algorithm runs in $O(|V| + |E|)$ time. \qquad
\end{proof}

\subsection{Remeshing}
\label{sec:remeshing}

In the node-nested case, the problem of remeshing may be treated as a procedure taking a mesh $M$ and some coarse subset
of $M$'s vertices $V^c$ and returning a new mesh $M^c$ that approximates the domain spanned by $M$ but only includes
vertices $V^c$.  Traditional constrained remeshing methods using standard meshing technology require specialized
information about the domain, including the location and topology of holes in the domain, that is not typically provided
with a pre-graded fine mesh.  This inhibits automation since a great deal of expert input is necessary to preserve the
boundary during coarsening.  Instead, we choose to locally remove the set of vertices $V \setminus V^c$ from $M$ one at
a time to create $M^c$.

In the 2D case, we remove a vertices from $M$ in-place by a simple series of geometric predicates and edge flips.  This
procedure is completely localized to the neighborhood around the vertex being removed.  This is done in a way that
retains the Delaunay condition~\cite{MostafaviDeleteInsert} if the initial mesh was Delaunay.  Similar but much more
complex techniques exist in 3D~\cite{LedouxFlipping}, but are difficult to make computationally robust enough for the
requirements of multigrid methods, where the creation of a single bad cell may severely impact the strength of the
resulting multigrid preconditioner.  In this work, we relax the Delaunay condition and instead focus on mesh quality of
the sort required by the multigrid method as stated in Section~\ref{sec:umg}.

An extremely simple method of removing a vertex $v$ from the mesh $M$ is by quality-conserving edge contraction.
Similar approaches have been proposed for this problem previously \cite{OllivierGoochContraction}.  This may be stated
as choosing $n^* \in N(v)$ such that for the set of cells $C_n$ that would be created during the contraction,
\begin{equation*}
  n^* = \operatorname*{\arg \min}\limits_{\substack{n \in N(v)}}(\displaystyle\max_{c \in C_n}AR(c))
\end{equation*}
subject to the constraint
\begin{equation*}
  \max_{c \in C_{n}}AR(c) < C_{AR}.
\end{equation*}
Each resulting $c$ must also be properly oriented in space.

The computational work done to remove a single vertex is $O(|C_v||N(v)|)$ as each potential configuration $C_n$ has
that, when compared to the set of cells adjacent to $v$, $C_v$, $|C_n| < |C_v|$ because at least two cells adjacent to
both $n$ and $v$ must be removed by each contraction. For each $n \in N(v)$ one must check the aspect ratio and
orientation of each cell in the associated $C_n$.  Assuming bounded mesh degree, the removal of $O(|V|)$ vertices would
take $O(|V|)$ time.

Note that we are directly enforcing the maximum aspect ratio that a given mesh modification may create, which makes this
akin to best effort remeshing.  If it is impossible to satisfy the quality conditions required of the mesh while
removing a vertex, the vertex is left in place, at least for the time being.  Such vertices are revisited for removal if
their link is modified by the removal of one of their neighbors, a process that may only happen $O(N(v))$
times, bounding the number of attempts that may be made to remove a given vertex.  Due to this, we cannot guarantee that
this 3D remeshing finds a tetrahedralization including just the coarse vertices. However, we note that our experience
with this method is that the number of vertices that must be left in place is significantly less than the typical number
of Steiner points added through direct Delaunay remeshing using \textit{Tetgen}~\cite{TetgenWeb}, and that this method
produces tetrahedralizations entirely suitable for use with multigrid methods.  The results of a study of the quality of
meshes created by this method is shown in Section~\ref{sec:quality}.

\subsection{Mesh Boundaries}
\label{sec:boundaries}

Preservation of the general shape of the domain is important to the performance of the multigrid method as the coarse
problem should be an approximation of the fine problem.  In the worst case a coarser mesh could become topologically
different from the finer mesh leading to very slow convergence or even divergence of the iterative solution.  If this is
to be an automatic procedure, then some basic criteria for the shapes of the sequence of meshes must be enforced.
Therefore, the vertex selection and remeshing algorithms must be slightly modified in order to take into account the
mesh boundaries.

First, we must define the features, in 2D and 3D, which require preservation.  We choose these features by use of
appropriate-dimensional notions of the curvature of the boundary.  Techniques like these are widely used in
computational geometry and computer graphics \cite{DynCurvature, RusinkiewiczCurvature}.  In 2D, the features we choose
to explicitly preserve are the corners of the mesh.  We consider any boundary vertex with an angle differing from a
straight line more than $C_{\mathcal{K}}$ to be a corner.  For the sake of this work we assume $C_{\mathcal{K}} =
\frac{\pi}{3}$.

In 3D, the discrete curvature at a boundary vertex is computed as the difference between $2\pi$ and the sum of the
boundary facet angles tangent to that vertex.  Vertices where the absolute value of the discrete curvature is greater
than $C_{\mathcal{K}}$ are forced to be included.  High-dihedral-angle edges, corresponding to ridges on the surface of
the domain, must also be preserved.  We consider any boundary edge with the absolute value of the dihedral angle greater
than $C_{\mathcal{K}}$ to be a boundary ridge.  

Our approach to protecting the boundary during the coarse vertex selection procedure is to separate vertex selection
into two stages, one for the interior, and one for the boundary.  In the interior stage, the boundary vertices are
marked as \texttt{included} automatically, and therefore any interior vertices violating the spacing condition with
respect to a boundary vertex are removed from $G_M$ (Figure \ref{fig:boundary_1}).  All boundary vertices now respect the
spacing condition with respect to all remaining interior vertices, making the second stage, boundary vertex selection,
entirely independent of the previous coarsening of the interior.  The boundary vertex selection procedure then operates
independently of this, producing the completely coarsened graph $G_{\tilde{M}_c}$ (Figure \ref{fig:boundary_3}).  In 3D
this process is repeated once again.  During the boundary coarsening procedure, vertices lying on edges that have been
identified as ridges are automatically \texttt{included}.  The ridge vertices are then coarsened in turn.  Corner
vertices are automatically \texttt{included} during all stages of vertex selection, as shown in Figure
\ref{fig:boundary_2}.
\begin{figure}
  \label{fig:boundaries}
  \centering
    \subfigure[interior coarsening] {
      \label{fig:boundary_1}
      \includegraphics[width=30mm]{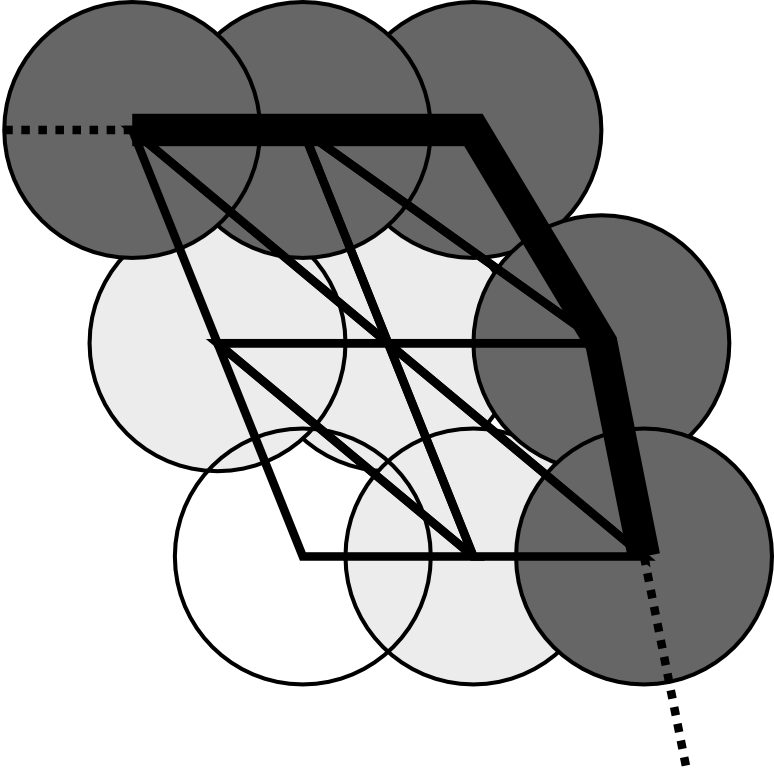}
    }
    \subfigure[boundary coarsening] {
      \label{fig:boundary_2}
      \includegraphics[width=30mm]{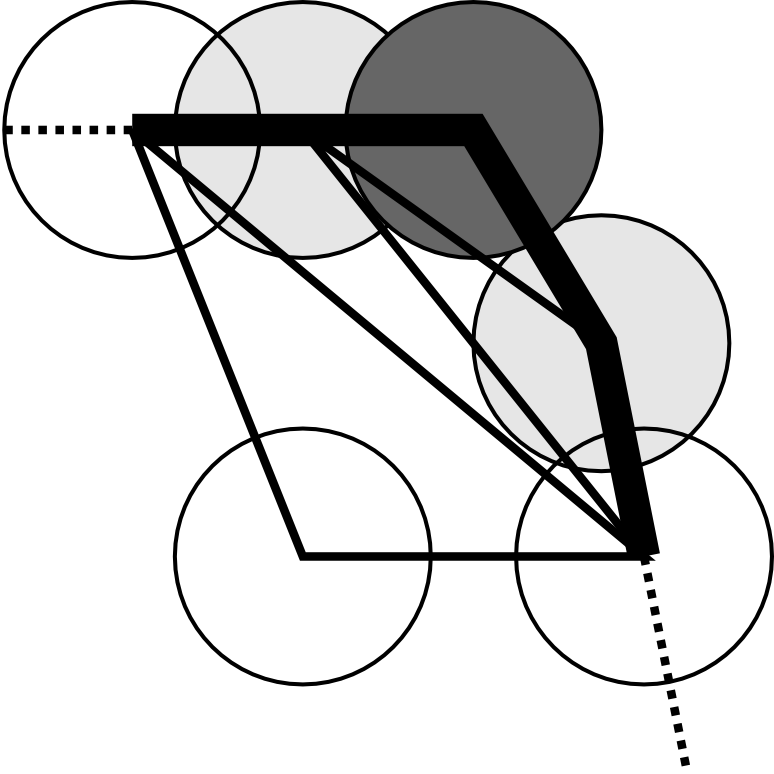}
    }
    \subfigure[final graph] {
      \label{fig:boundary_3}
      \includegraphics[width=30mm]{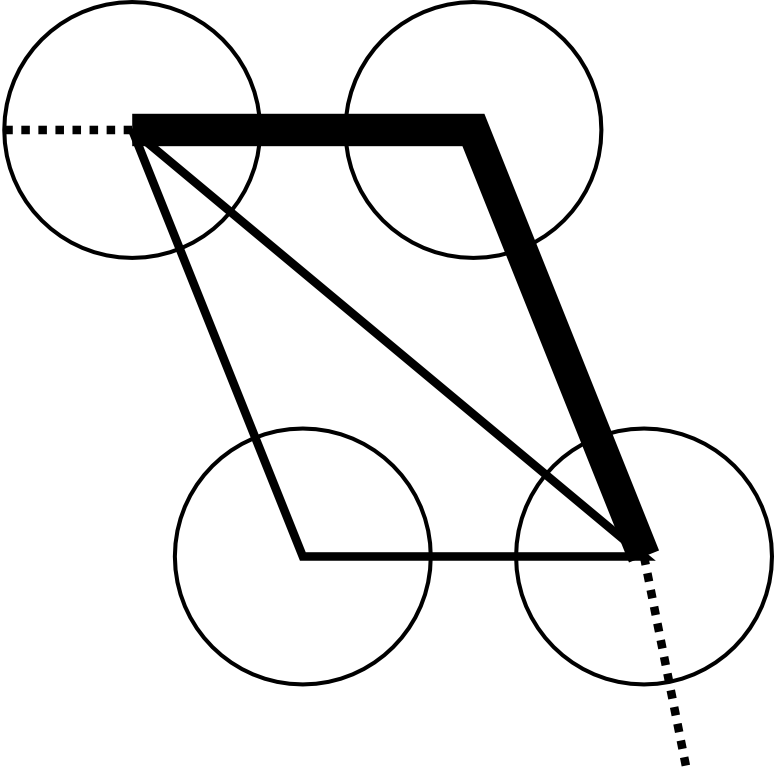}
    }
    \caption{Interior and boundary coarsening stages.  Dark gray vertices are \texttt{included} automatically; Light
      gray circles violate the spacing condition.}
\end{figure}
For remeshing, the only modification necessary for boundary preservation is that the removal of a boundary or ridge
vertex may only be done by contraction along an edge lying on the boundary or ridge.  This simple modification works
especially well for meshes that have a lot of flat planes on their surfaces with sharp ridges separating them.  Without
this procedure, ridges may appear torn or flipped in the coarser meshes, changing the shape of the domain significantly.
As the hierarchy coarse meshes are created, new edges or vertices may become marked as corners or as belonging to ridges.
\section{Interpolation Operators}
\label{sec:interpolation}
We have shown that local traversal is a powerful tool for the construction of node-nested coarse meshes.  Traversal of
the node-nested meshes may also be used to efficiently construct the interpolation operators $I_{l}$ for $l \in [0,
n-1]$.  $I_{l}$ is defined as the interpolation operator between the finite element function spaces $V^h_{c}$ and
$V^h_{f}$, which are defined over the coarse mesh $M^c$ and fine mesh $M^f$ respectively for the adjacent $f = l$ and $c
= l+1$.  We also assume that the domain $\Omega$ spanned by all $M$ is connected.  For the purposes of this paper, we
assume that for each DoF $i$, associated with basis function $\phi^{f}_i$ in $V^f_h$ in the system has some nodal point
$x_i \in \mathbb{R}^d$ and that a function $f$ may be projected into $V^f_h$ by constructing the coefficients
\begin{equation*}
  f_i = f(x_i)\cdot\psi^{f}_i(x_i).
\end{equation*}
Given this assumption, we may construct the prolongation operator from functions in $V^c_h$ to $V^f_h$ by
associating each $i$ with a cell $\tau^{c}_k$ in $M^{c}$ and constructing, for $\psi^{f}_i$ and basis functions
$\psi^{c}_j$ supported on $\tau^{c}_k$:
\begin{equation*}
  \label{eq:interpentry}
  I_{ij} = \psi^{c}_j(x_i)\cdot\psi^{f}_i(x_i).
\end{equation*}
One can, for the sake of simplicity, associate every $\psi^{f}_j$ with a single cell $\tau^{f}_i$. The choice of cell is
arbitrary between the cells the unknown is supported on.  This makes the problem equivalent to finding, for some set of
$x \in \Omega$, $T^{c}(x)$, defined as the cell in $M^{c}$ that $x$ is located in.  We also define $\tilde
T^{c}(x)$, which is some cell in $M^{c}$ that is, in some sense, nearby $x$.  In addition to the points $x_i$ associated
with all $\psi^{f}$, we have the midpoint $x_{\tau^f_i}$ of every cell $\tau^f_i \in M^{f}$.

We define a two-level nested-loop traversal algorithm.  The outer loop locates $T^{c}(x_{\tau^f})$ for each $\tau^f \in
M^{f}$, and the inner loop determines $T^{c}(x_i)$ for all $x_i$ associated with $\psi^{f}_i$ supported on $\tau^{f}_i$.
Once $T^{c}(x_i)$ is resolved for all $i$, one may construct the nodal interpolation operator.

The outer loop consists of breadth-first searches (BFSes) on the graph of the neighbor relations of the cells of $M^{c}$
in order to determine $T^{c}(x_{\tau_k})$ for each $\tau_k \in M^{f}$.  This is implemented by a simple FIFO queue in
which the neighbors of a given cell are pushed on the back of the queue when it is visited.  Enqueued and visited cells
are marked as such and not enqueued again.  We say that two cells $\tau^{c}_0$ and $\tau^{c}_1$ are neighbors if they
share a vertex, edge, or facet.  The BFS to locate $\tau_k$ starts with the cell $\tilde{T}^{c}(x_{\tau_k})$.

As $T^{c}(x_{\tau^f_k})$ is established for each cell $\tau^f_k$, the inner loop is invoked.  This loop consists of a BFS
for each $\psi^f_i$ associated with $\tau^f_k$ and determines $T^{c}(x_i)$ by BFS starting from $\tilde{T}^{c}(x_i) =
T^{c}(x_{\tau^f_k})$.

The last ingredient in the algorithm is how to determine $\tilde{T}^{c}(x_{\tau^f_k})$. One simple way of doing this is
setting $\tilde{T}^{c}(x_{\tau^f_k}) = T^{c}(x_{\tau^f_m})$ for any cell $\tau^f_m$ that is neighboring $\tau^f_k$.
This notion of locality may be exploited by setting the values of $\tilde{T}^{c}(\tau_n)$ for all neighbors $\tau_n$ of
a cell $\tau$ upon determining $T^{c}(\tau)$.  When $T^{c}(\tau^f)$ for any $\tau^f \in M^f$ is determined, the
connectivity of $M^{c}$ and $M^{f}$ may be effectively traversed in tandem in order to extend the search over the whole
meshed domain.

This process may be both started and made more efficient by exploiting the node-nested properties of the meshes.  We
have that meshes $M^{f}$ and $M^{c}$ will share some set of vertices $V^{(f, c)}$ due to our node-nested coarsening
procedure.  Define $\tau^{f}_v$ to be a cell in mesh $M^{f}$ that is adjacent to some $v \in V^{(f, c)}$ and
$\tau^{c}_v$ to be an arbitrary cell in $M^{c}$ adjacent to that same $v$.  Then, one may initialize
$\tilde{T^{c}}(x_{\tau^{f}_v})$ to be $\tau^{c}_v$.

An obervation about the complexity of this algorithm may be established assuming that the conditions in Section~\ref{sec:umg}
are satisfied by the hierarchy of meshes.  It should be obvious that the complexity is going to be bounded
by the total number of BFSes done during the course of the algorithm multiplied by the worst-case complexity of an
invocation of the BFS.  It's easy to see that for $|M^{l}|$ cells in the fine mesh and $n$ unknowns, the geometric
search must be done $|M^{l}| + n$ times.


In order to bound the complexity of a given BFS to a constant, we must show how many steps of traversal one must search
outwards in order to locate some $T^{c}(\tau^f)$ given $\tilde{T}^{c}(\tau^f)$.  This may be accomplished by showing
that the length scale the search has to cover on the fine mesh may only contain a given number of coarse mesh cells.  We
know that the $dist(x_{\tau^f_0}, x_{\tau^f})$ is going to be less than some maximum search radius $r_{\tau^f} =
h_{\tau^f} + h_{\tau^f_0}$ given that they are adjacent.  We also may put a lower limit on the minimum length scale of
cells in $M^c$ that overlap $\tau^f$ and $\tau^f_0$ by the constant in Eq.~\ref{eq:loccomp2} and by the aspect ratio
limit in Eq.~\ref{eq:aspectratiobound} as $h^c_{\min} = \frac{C_0h_{\tau^f}}{C_{AR}}$.

This gives us that the number of cells that may fit between $\tilde{T}^c(\tau^f)$ and $T^{c}(\tau^f)$ is at most
$\left\lceil\frac{r_\tau^f}{h^c_{\min}}\right\rceil$, which is independent of overall mesh size.  Note that the distance
between the center of a cell $\tau$ and some nodal points $x_i$ of $\psi_i$ supported on $\tau$ is always less than or
equal to $h_\tau$.  This extends this constant-size traversal bound to the inner loop also, making the entire algorithm
run in $O(|M^f| + n)$ time. In practice on isotropic meshes $\tilde{T}(\tau^f)$ and $T(\tau^f)$ are almost always in the
same cell or one cell over, meaning that the topological search is an efficient method for building the interpolation
operators.

An extension of the algorithm that has proven necessary on complex meshes is to allow for interpolation to $x_i$ not
located within a cell of the mesh.  For example, the Pacman mesh (Fig.~\ref{fig:pacmanhierarchy}), when coarsened, will
have vertices that lie outside the coarser mesh on the round section.  In order to do this, the outer loop BFSes are
replaced by a more complex arrangement where the BFSes search for the nearest cell rather than an exact location.  This
nearest cell then becomes $T(x_\tau)$.  The inner loop is replaced by a similar procedure, which then projects any $x_i$
that could not be exactly located to some $\tilde{x}_i$ on the surface of the nearest cell and modifies
the basis function projection to use $\phi^c_j(\tilde{x}_i)$ instead of $\phi^c_j(x_i)$.

This procedure is easily extended to discretizations consisting of non-nodal DoFs by locating sets of quadrature points
associated with the fine mesh cells instead of DoF nodes.  This also allows for the construction of pseudointerpolants
satisfying smoothness assumptions, such as the Scott-Zhang type interpolants~\cite{ScottZhangInterpolant}.

\section{Model Problems}
\label{sec:problems}

We test the multigrid method using standard isotropic non-quasi-uniform examples from the literature.  The domains used
are the Pacman in 2D and the Fichera corner in 3D.  The problems and boundary conditions computed are designed to induce
a corner singularity that makes refinement such as that described in Section~\ref{sec:afem} necessary.

The standard model problem for the Pacman domain is \cite{ApelComparison} designed to have a pollution effect at the
corner singularity.  The Pacman domain is a unit circle with a tenth removed.  The reentrant angle is therefore $\theta_r
= \frac{9\pi}{5}$ radians.  The associated differential equation is
\begin{align*}
  - \Delta u &= 0 \mbox{ \text{ in } } \Omega,\\
    u(r, \theta) &= r^{\frac{2}{3}}\sin\left(\frac{2}{3}\theta\right)  \mbox{\text{ on }} \partial\Omega.
\end{align*}
For the Fichera corner model problem, we separate the boundary of the domain into the reentrant surface $\partial
\Omega_{i}$ and the outer surface $\partial \Omega_{o}$.  The Fichera corner mesh consists of a twice-unit cube centered
at the origin with a single octant removed.  The reentrant angle is $\theta_r = \frac{3\pi}{2}$ radians along the
reentrant edges.  A standard model problem with both Dirichlet and Neumann boundary conditions used for the Fichera
corner domain is \cite{RachowiczAutoHP}
\begin{align*}
  -\Delta u &= 0 \mbox{\text{ in } } \Omega,\\
  \frac{\partial u}{\partial n} &= g \text{ on } \partial\Omega_{o},\\
  u &= 0 \text{ on } \partial\Omega_{i}.
\end{align*}
Where 
\begin{align*}
    g(x,y,z) &= g_{xy} + g_{yz} + g_{zy},\\
    g_{x_ix_j}(r, \theta) &= r^{\frac{2}{3}}\cos(\theta).
\end{align*}
where $r$ and $\theta$ consist of the radius and angle when restricted to the $x_i$, $x_j$ plane.  Unfortunately, this
problem does not have an analytical exact solution.  Our goal is to show that we can effectively coarsen the \textit{a
  priori} grading required to resolve the problem, so we will not be concerned with the solution accuracy as a measure
of performance.
\subsection{Mesh Quality Study}
\label{sec:quality}

In order to motivate the use of the method with the multigrid solver, one must look if our implementation of the method
actually lives up to the mesh quality metrics stated as requirements for guaranteed multigrid performance.  This study
was done using large initial Pacman and Fichera corner meshes coarsened using the same algorithm used for the multigrid
studies in Section~\ref{sec:multperf}.

The measurements of mesh quality we have taken correspond to the bounds on the meshes in the multigrid requirements.
These are the number of vertices and cells in the mesh, the worst aspect ratio in each mesh, the maximum number of cells
in a mesh overlapping any given cell in the next coarser mesh, and the maximum change in local length-scale between a
mesh and the next coarser mesh at any point in the domain.  The meshes are created by repeated application of the
coarsening algorithm with $\beta = 1.5$ in 2D and $\beta = 1.8$ in 3D.
\begin{figure}[H] 
  \centering
  \includegraphics[width=140mm]{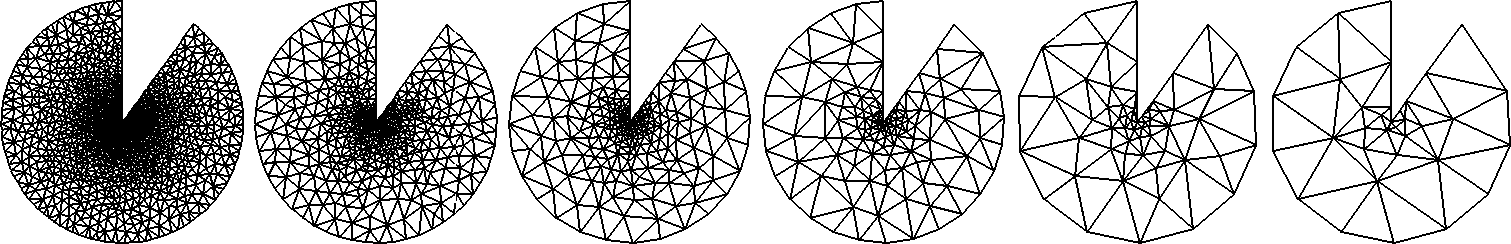} 
  \caption{Mesh hierarchy produced by function-based coarsening of an initial pacman mesh.}
  \label{fig:pacmanhierarchy}
\end{figure}
\begin{table}[H]
\caption{Hierarchy quality metrics - 2D}
\begin{center}
\begin{tabular}{|r||r|r|r|r|r|}
\hline
\multicolumn{6}{|c|}{Pacman Mesh, $\beta = 1.5$}\\
\hline
level & cells & vertices & $\max(AR(\tau))$ &  $\displaystyle\sup_{\tau \in M^k}\{|S^k_{\tau}|\}$ & $\displaystyle\max_{x \in \Omega}\left(\frac{h_{k-1}(x)}{h_{k}(x)}\right)$ \\
\hline
0 & 70740 & 35660 & 3.39 & -- & -- \\
1 & 22650 & 11474 & 7.64 & 14 & 9.48 \\
2 & 7562  & 3858  & 7.59 & 15 & 6.23 \\
3 & 2422  & 1254  & 4.63 & 15 & 4.19 \\
4 & 811   & 428   & 5.52 & 15 & 5.32 \\
5 & 257   & 143   & 5.94 & 15 & 8.51 \\
6 & 86    & 52    & 7.76 & 12 & 4.94 \\
\hline
\end{tabular}
\end{center}
\label{tab:coarsenperf2d}
\end{table}
\begin{figure}[H] 
  \centering
  \includegraphics[width=140mm]{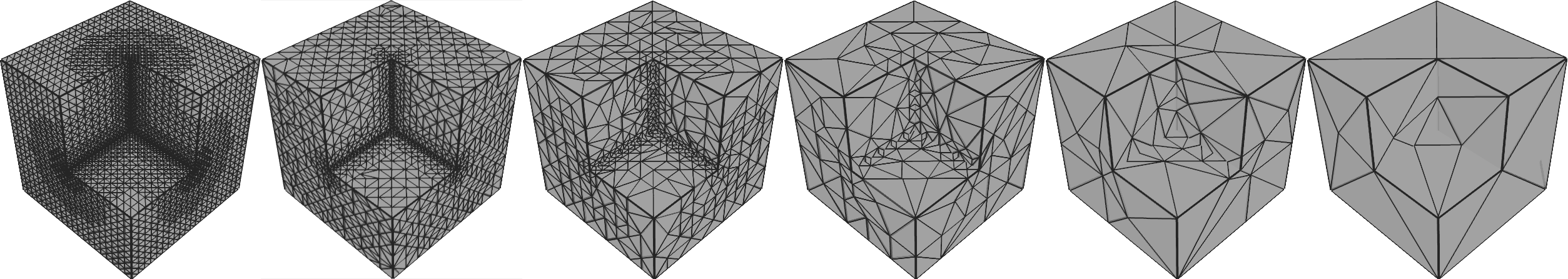} 
  \caption{hierarchy of Fichera meshes}
  \label{fig:ficherahierarchy}
\end{figure}
\begin{table}[H]
\caption{Hierarchy quality metrics - 3D}
\begin{center}
\begin{tabular}{|r||r|r|r|r|r|}
\hline
\multicolumn{6}{|c|}{Fichera Mesh, $\beta = 1.8$, $C_{AR} = 60$}\\
\hline
level & cells & vertices & $\max(AR(\tau))$ & $\displaystyle\sup_{\tau \in M^k}\{|S^k_{\tau}|\}$ & $\displaystyle\max_{x \in \Omega}\left(\frac{h_{k-1}(x)}{h_{k}(x)}\right)$ \\
\hline
0 & 373554 & 72835 & 4.41 & --  & --   \\
1 & 49374  & 9120  & 59.3 & 197 & 6.63 \\
2 & 11894  & 2269  & 59.2 & 131 & 6.70 \\
3 & 3469   & 693   & 59.3 & 94  & 6.68 \\
4 & 914    & 208   & 58.8 & 121 & 6.61 \\
5 & 182    & 52    & 49.5 & 94  & 6.87 \\
\hline
\end{tabular}
\end{center}
\label{tab:coarsenperf3d}
\end{table}
In 2D (Table \ref{tab:coarsenperf2d}) the aspect ratio of the resulting cells stays within acceptable limits during the
entire process.  The slight increase in the aspect ratio is expected for the coarsening of highly graded meshes, as the
coarser versions must necessarily remove vertices lying between refined and coarse regions of the non-quasi-uniform
mesh.  However, the associated increase in aspect ratio is kept to a minimum by the enforcement of the Delaunay
condition.

In 3D (Table \ref{tab:coarsenperf3d}), we see consistent decrease of the mesh size and increase in the length scale.
The maximum aspect ratio stays around $C_{AR}$ for most of the levels.  Further work on our incredibly simple remeshing
scheme should be able to improve this.  However, we do not see successive degradation of the quality of the series of
meshes as the coarsening progresses.  We can assume that the quality constraints are reasonably met in both 2D and 3D.

\subsection{Multigrid Performance}
\label{sec:multperf}
\begin{table}
  \centering 
  \caption{
    Multigrid Performance on the 2D (Pacman) and 3D (Fichera) problems. Levels is the number of mesh levels. MG is
    the number of multigrid V-cycles. ILU is ILU-preconditioned GMRES iterations. $L_2$ Error is computed based upon
    difference from the exact solution for the Pacman test problem.
} 
  \label{tab:mgperf}
  \begin{minipage}{0.45\textwidth}
  \subtable[Pacman Problem Performance]{
    \centering
    \begin{tabular}{|r||r|r|r|r|}
      \hline
      DoFs        & levels    & MG            & ILU              & $L_2$ Error \\
      \hline
      762         & 3         & 6             & 76               & 9.897e-6 \\
      1534        & 4         & 6             & 129              & 3.949e-6 \\
      3874        & 5         & 6             & 215              & 9.807e-7\\
      9348        & 5         & 7             & 427              & 2.796e-7\\
      23851       & 6         & 7             & 750              & 1.032e-7 \\
      35660       & 7         & 7             & 1127             & 8.858e-8 \\
      78906       & 7         & 7             & 1230             & 3.802e-8 \\
      139157      & 8         & 9             & 3262             & 1.535e-8 \\
      175602      & 8         & 9             & 3508             & 3.440e-9 \\
      \hline
    \end{tabular}
    \label{tab:mgperf2d}
  }
\end{minipage}
  \begin{minipage}{0.45\textwidth}
  \subtable[Fichera Problem Performance]{
    \centering
      \begin{tabular}{|r||r|r|r|}
        \hline
        DoFs     &    levels & MG              &ILU               \\
        \hline
        2018        & 2         & 7               & 46               \\
        3670        & 2         & 8               & 60               \\
        9817        & 3         & 8               & 112              \\
        27813       & 4         & 9               & 207              \\
        58107       & 5         & 9               & 385              \\
        107858      & 6         & 9               & 543              \\
        153516      & 6         & 9               & 440              \\
        206497      & 7         & 9               & 616              \\
        274867      & 8         & 9               & 652              \\
        \hline
      \end{tabular}
      \label{tab:mgperf3d}
    }
    \end{minipage}
\end{table}
The performance of multigrid based upon the mesh series created by this algorithm using the two test examples was
carried out using the DOLFIN \cite{DOLFIN} finite element software modified to use the PCMG multigrid preconditioner
framework from PETSc \cite{PETScWeb}. The operators were discretized using piecewise linear triangular and tetrahedral
finite elements.  The resulting linear systems were then solved to a relative tolerance of $10^{-12}$.  The solvers used
were ILU-preconditioned GMRES for the standard iterative case, shown in the ILU columns as a control.  In the
multigrid case we chose to use GMRES preconditioned with V-cycle multigrid.  Three pre and post-smooths using ILU as the
smoother were performed on all but the coarsest level, for which a direct LU solve was used.  For this to make sense, the
coarsest mesh must have a small, nearly constant number of vertices; a condition easily satisfied by the mesh creation
routine.  We coarsen until the coarsest mesh is around 200 vertices in 2D or 300 vertices in 3D.


These experiments show that the convergence of the standard iterative methods becomes more and more arduous as the
meshes become more and more singularly refined.  The singularity in 2D is much greater, due to the sharper reentrant
angle, than it is in 3D, so the more severe difference in performance between multigrid and the control method is to be
expected.  We see that the number of multigrid cycles levels out to a constant for both 2D (Table \ref{tab:mgperf2d})
and 3D (Table \ref{tab:mgperf3d}).  We also see that a steadily increasing number of multigrid levels are automatically
generated and that the method continues to behave as expected as the coarsening procedure is continually applied.

\section{Discussion}

We have proposed and implemented a method for mesh coarsening that is well suited to the construction of geometric
unstructured multigrid on 2D and 3D meshes.  The technique used to construct the mesh hierarchy is novel in that it
combines the advantages of graph-based mesh coarsening algorithms with those of techniques considering the local feature
size of the mesh.  This is done in a straightforward way that only considers the local connectivity of the mesh.

The further extension of the method to a parallel environment is fairly straightforward.  Incorporating a separate
parallel coarsening stage for vertices on the processor boundaries would be a natural extension of the method.
Techniques with a localized coarser solve, such as a variant of the Bank-Holst method \cite{BankBankHolstVariants},
could be easily replicated in this framework as well.

Meshing issues still exist, as the local remeshing can be quite costly in 3D.  There are several ways the remeshing
could be improved, both in terms of efficiency and mesh quality guarantees. It is an experimental observation that
remeshing using meshing software with constrained interior vertices is often faster than vertex removal.  A mixed
approach, where the domain could be topologically traversed and divided into regions that could be remeshed separately,
may be an interesting compromise.  Both approaches to remeshing are also easily extendable to the creation of an
all-encompassing parallel algorithm.

Despite this cost, the big win is twofold: treating the construction of coarse spaces as a precomputation step in some
nonlinear or optimization solve, and using higher-order elements, where the topological description of the mesh is
smaller compared to the number of DoFs.  It's been shown to be quite important to preserve the order of the
approximation space when using multigrid for higher-order problems \cite{KoesterHigherOrderMG}, and the techniques for
efficiently building both pointwise and other types of interpolants shown here already work in a similar fashion for
higher order FEM discretizations.  As we can easily control the rate at which the mesh is decimated, we can coarsen more
aggressively for the high order methods, retaining the optimal complexity of the method.

One final future direction is to look at how this techique applies to problems on anisotropic meshes and the associated
theory.  There are successful node-nested technques proposed for semicoarsening anisotropic meshes
\cite{mavriplisnodenestedMG}, and the coarsening technique described here could be applied in interesting ways.  The
other related problem of interest involves multigrid methods for problems with highly variable coefficients.  Often
these methods are developed and implemented in some nested framework \cite{ChenLocalMultigrid}, but with our ability to
build meshes and interpolation operators that preserve interfaces (due to their construction by traversal) we should be
able to look at these problems in the generalized unstructured setting.


\end{document}